\documentclass{amsart}
\usepackage{makecell}
\usepackage{placeins}
\usepackage{amssymb}
\usepackage{hyperref}

\newtheorem{proposition}{Proposition}[section]

\newtheorem{lemma}[proposition]{Lemma}

\newtheorem{theorem}[proposition]{Theorem}

\newcommand{\N}{\mathbb{N}}

\begin{document}

\title[Bounding the number of classes in terms of a prime]{Bounding the number of classes of a finite group in terms of a prime}
\author{Attila Mar\'oti}
\address[Attila Mar\'oti]{Alfr\'ed R\'enyi Institute of Mathematics, Hungarian Academy of Sciences, Re\'altanoda utca 13-15, H-1053, Budapest, Hungary}
\email{maroti.attila@renyi.mta.hu}
\thanks{The work of the first author leading to this application has
received funding from the European Research Council (ERC) under the
European Union's Horizon 2020 research and innovation programme
(grant agreement No. 741420). The first author was supported by the National Research, Development and Innovation Office
(NKFIH) Grant No.~K132951 and No.~K115799 and by the J\'anos Bolyai Research Scholarship of the Hungarian Academy of
Sciences.}
\author{Iulian I. Simion}
\address[Iulian I. Simion]{Department of Mathematics\\
`Babe\c s-Bolyai' University\\
str. Ploie\c sti nr.23-25, 
400157 Cluj-Napoca\\
Romania}
\email{simion@math.ubbcluj.ro}
\date{\today}

\begin{abstract}
H\'ethelyi and K\"ulshammer showed that the number of conjugacy classes $k(G)$ of any solvable finite group $G$ whose order is divisible by the square of a prime $p$ is at least $(49p+1)/60$. Here an asymptotic generalization of this result is established. It is proved that there exists a constant $c>0$ such that for any finite group $G$ whose order is divisible by the square of a prime $p$ we have $k(G) \geq cp$.
\end{abstract}
\maketitle

\section{Introduction}

Let $k(G)$ denote the number of conjugacy classes of a finite group $G$. This is also the
number of complex irreducible characters of $G$. Bounding $k(G)$ is a fundamental
problem in group and representation theory.

Let $G$ be a finite group and $p$ a prime divisor of the order $|G|$ of $G$. In this paper we discuss lower bounds for $k(G)$ only in terms of $p$.

Pyber observed that results of Brauer \cite{Brauer} imply that $G$ contains at least $2\sqrt{p - 1}$ conjugacy classes provided that $p^{2}$ does not divide $|G|$. Building on works of H\'ethelyi and K\"ulshammer \cite{HK}, Malle \cite{Malle}, Keller \cite{Keller}, H\'ethelyi, Horv\'ath, Keller and Mar\'oti \cite{HHKM}, it was shown in \cite{attilaLB} that $k(G)\geq 2\sqrt{p-1}$ for any finite group $G$ and any prime $p$ dividing $|G|$, with equality if and only if $\sqrt{p-1}$ is an integer, $G=C_p\rtimes C_{\sqrt{p-1}}$ and $C_G(C_p)=C_p$.

The objective of the current paper is to provide a stronger lower bound for $k(G)$ in case $p^2$ divides $|G|$. H\'ethelyi and K\"ulshammer \cite{HK2} showed that for any finite solvable group $G$ and any prime $p$ such that $p^{2}$ divides $|G|$, the number of conjugacy classes of $G$ is at least $(49p+1)/60$. This bound is sharp \cite{HK2} for infinitely many primes $p$, however it does not generalize \cite{HK} to arbitrary finite groups since there are infinitely many non-solvable groups $G$ and primes $p$ with $k(G) = 0.55 p - 0.05$. 

The main result of this paper is the following.

%

\begin{theorem}
  \label{main_theorem}
  There exists a constant $c>0$ such that for any finite group $G$ whose order is divisible by the square of a prime $p$ we have $k(G)\geq cp$.
  \end{theorem}


Questions of Pyber and the papers \cite{HK} and \cite{HK2} of H\'ethelyi and K\"ulshammer motivated our result. 

Let $B$ be a $p$-block of a finite group $G$ and let $D$ be a defect group of $B$. The number $k(B)$ of complex irreducible characters of $G$ associated to the block $B$ is a lower bound for $k(G)$. A recent result of Otokita \cite[Corollary 4]{Otokita} states that $k(B) \geq (p^{m}+p-2)/(p-1)$ where $p^{m}$ denotes the exponent of the center of $D$. 

Finally, note that Kov\'acs and Leedham-Green \cite{KLG} have constructed, for every odd prime $p$, a finite $p$-group $G$ of order $p^p$ with $k(G) = \frac{1}{2}(p^{3} - p^{2} + p + 1)$ (see also \cite{Pyber}).

\section{Affine groups}
\label{affine_groups}

The purpose of this section is to prove Proposition \ref{hard}. For this we need the following lemma. The base of the logarithms in this paper is always $2$.  

\begin{lemma}
\label{o(1)}
Let $H$ be a finite group and $V$ be a finite, faithful, completely reducible $H$-module over a finite field of characteristic $p$. Assume that $H$ has no composition factor isomorphic to an alternating group of degree larger than ${(\log p)}^{3}$ and has no composition factor isomorphic to a simple group of Lie type defined over a field of characteristic $p$. Put $p^{n} = |V|$. Then $H$ has an abelian subgroup of index at most ${(c_{1} \log p)}^{7(n-1)}$ for some universal constant $c_{1} > 1$.  
\end{lemma} 

Note that once Lemma \ref{o(1)} is proved it may be extended by a theorem of Chermak and Delgado \cite[Theorem 1.41]{Isaacs} as follows. Under the conditions of Lemma \ref{o(1)}, the group $H$ contains a characteristic abelian subgroup of index at most ${(c_{1} \log p)}^{14(n-1)}$ for some universal constant $c_{1} > 1$ .

\begin{proof}[Proof of Lemma \ref{o(1)}]
Assume first that $V$ is a primitive and irreducible $H$-module. We use the following structure result which is implicit in the proofs of \cite{GMP} (see for example the proof of \cite[Theorem 9.1]{GMP}). Let $F$ be the largest field such that $H$ embeds in $\Gamma L_F(V)$. Let $C$ be the subgroup of non-zero elements in $F$. We claim that $|H/(H \cap C)| \leq {(c_{1} \log p)}^{7(n-1)}$ for some universal constant $c_{1} > 1$. For this we may assume that $C\leq H$. 

Let $H_0$ be the centralizer of $C$ in $H$ and let $R$ be a normal subgroup of $H$ contained in $H_0$ minimal with respect to not being contained in $C$ (if such exists). There are two possibilities for $R$. It is of symplectic type and $|R/Z(R)|=r^{2a}$ for some prime $r$ and integer $a$ such that $r$ divides $|F|-1$ or $R$ is a central product of $t$ isomorphic quasisimple groups.

Choose a maximal collection $J_1, \ldots, J_m$ of such non-cyclic normal subgroups in $H_0$ which pairwise commute (if such exist). Let $J$ be the central product of the subgroups $J_{1}, \ldots , J_{m}$. Then $H_0/(C \cdot \mathrm{Sol}(J))$ embeds in the direct product of the automorphism groups of $J_i/Z(J_i)$ where $\mathrm{Sol}(J)$ denotes the solvable radical of $J$. (Note that in the proof of \cite[Theorem 9.1]{GMP} it was falsely asserted that $H_{0}/C$ embeds in the direct product of the automorphism groups, however this did not affect the proof of \cite[Theorem 9.1]{GMP} nor \cite[Theorem 10.1]{GMP}.)

Let $W$ be an irreducible constituent of $V$ for the normal subgroup $J$ of $H$ (provided that $J$ is non-trivial). Since $H$ is primitive on $V$, it follows that $J$ acts homogeneously on $V$ by Clifford's theorem. Let $E = \mathrm{End}_{FJ}(W)$. Now $W \cong U_{1} \otimes \cdots \otimes U_{m}$ where $U_{i}$ is an absolutely irreducible $EJ_{i}$-module by \cite[Lemma 5.5.5]{KL}. Notice that $E$ may be viewed as a subfield of $\mathrm{End}_{FJ}(V)$ and since $J$ is normal in $H$, the multiplicative group of $E$ is normalized by $H$. Our choice of $F$ implies that $E = F$. If $J_{i}$ is of symplectic type with $J_{i}/Z(J_{i})$ of order $r_{i}^{2a_{i}}$, then $\dim U_{i} = r_{i}^{a_{i}}$. If $J_{i}$ is a central product of $t$ isomorphic quasisimple groups $Q_{i,j}$ with $1 \leq j \leq t$, then $U_{i} \cong U_{i,1 } \otimes \cdots \otimes U_{i,t}$ where $U_{i,j}$ is an absolutely irreducible (faithful) $FQ_{i,j}$-module for every $j$ with $1 \leq j \leq t$, by \cite[Lemma 5.5.5 and Lemma 2.10.1]{KL}.        

Let $|F| = p^{f}$ and let $d = \dim_{F} V$. The product of the orders of all abelian composition factors in any composition series of the factor group $H/C$ is less than $f \cdot d^{2 \log d + 3} \leq n^{2 \log n + 4}$ by \cite[Theorem 10.1]{GMP} and its proof. This is at most ${(c_{2} \log p)}^{n-1}$ for some constant $c_{2} > 2$. We may now assume that $J \not= 1$ and $n>1$.

Let $b(X)$ denote the product of the orders of all non-abelian composition factors in any composition series of a finite group $X$. Since $|H/C| \leq {(c_{2} \log p)}^{n-1} b(H)$, we proceed to bound $b(H)$. 

Without loss of generality, assume that $J_{1}, \ldots , J_{k}$ are groups of symplectic type with $k \geq 0$ and $|J_{i}/Z(J_{i})| = r_{i}^{2a_{i}}$ for some primes $r_i$ and integers $a_{i}$, and assume that $J_{k+1}, \ldots , J_{m}$ are groups not of symplectic type. For each $\ell$ with $k+1 \leq \ell \leq m$, let $J_{\ell}$ be a central product of $t_{\ell}$ copies, say $Q_{\ell,1}, \ldots , Q_{\ell,t_{\ell}}$, of a quasisimple group $Q_{\ell}$. In this case $U_{\ell} \cong U_{\ell,1 } \otimes \cdots \otimes U_{\ell,t_{\ell}}$ where $U_{\ell,j}$ is an irreducible (faithful) $Q_{\ell,j}$-module for every $j$ with $1 \leq j \leq t_{\ell}$. Using this notation we may write the following. 
\begin{equation}
\label{eee1}
\begin{split}
n \geq d = \dim V \geq \dim W = \Big( \prod_{i=1}^{k} \dim U_{i} \Big) \cdot \Big( \prod_{\ell=k+1}^{m} \dim U_{\ell} \Big) = \\ = \Big( \prod_{i=1}^{k} {r_{i}}^{a_{i}} \Big) \cdot \Big( \prod_{\ell = k+1}^{m} {(\dim U_{\ell,1})}^{t_{\ell}} \Big) \geq \Big( \prod_{i=1}^{k} {r_{i}}^{a_{i}} \Big) \cdot 2^{\sum_{\ell = k+1}^{m} t_{\ell}}.
\end{split}
\end{equation}

Since $H/H_{0}$ and $C \cdot \mathrm{Sol}(J)$ are solvable, $b(H) = b(H_{0}/(C \cdot \mathrm{Sol}(J)))$. Recall from the third paragraph of this proof that the group $H_{0}/(C \cdot \mathrm{Sol}(J))$ embeds in the direct product of the automorphism groups of the $J_{i}/Z(J_{i})$. There exists a chain of subnormal subgroups 
$$H_{0}/(C \cdot \mathrm{Sol}(J)) = N_{0} \triangleright N_{1} \triangleright \cdots \triangleright N_{m} = \{ C \cdot \mathrm{Sol}(J) \}$$ such that $N_{i-1}/N_{i} \leq \mathrm{Aut}(J_{i}/Z(J_{i}))$ for every $i$ with $1 \leq i \leq m$. These give 
\begin{equation}
\label{ee1}
b(H) \leq \Big( \prod_{i=1}^{k} |N_{i-1}/N_{i}| \Big) \cdot \Big( \prod_{\ell = k+1}^{m} b(N_{\ell-1}/N_{\ell}) \Big). 
\end{equation}

Since $\prod_{i=1}^{k} r_{i}^{a_{i}} \leq n$ by (\ref{eee1}), we have
\begin{equation}
\label{ee2}
\prod_{i=1}^{k} |N_{i-1}/N_{i}| < \prod_{i=1}^{k} r_{i}^{4 a_{i}^{2}} \leq \prod_{i=1}^{k} n^{4 \log (r_{i}^{a_{i}})} \leq n^{4 \sum_{i=1}^{k} \log (r_{i}^{a_{i}})} \leq n^{4 \log n}.
\end{equation}

We see by Schreier's conjecture that for every $\ell$ with $k+1 \leq \ell \leq m$, we have $b(N_{\ell-1}/N_{\ell}) \leq |T_{\ell}| \cdot {b(Q_{\ell})}^{t_{\ell}}$ where $T_{\ell}$ is some permutation group of degree $t_{\ell}$ having no composition factor isomorphic to an alternating group of degree larger than ${(\log p)}^{3}$. Now $|T_{\ell}| \leq {(2 \log p)}^{3(t_{\ell}-1)}$ by \cite[Corollary 1.5]{attila2002}. Using the fact that $\sum_{\ell = k+1}^{m} t_{\ell} \leq \log n$ (see (\ref{eee1})), we have
\begin{equation}
\label{ee3}
\prod_{\ell = k+1}^{m} b(N_{\ell-1}/N_{\ell})  \leq {(2 \log p)}^{3(\log n -1)} \cdot \Big( \prod_{\ell = k+1}^{m}{b(Q_{\ell})}^{t_{\ell}} \Big).
\end{equation}

It follows by (\ref{ee1}), (\ref{ee2}) and (\ref{ee3}) that 
\begin{equation}
\label{ee4}
b(H) < {(c_{3} \log p)}^{3 (n-1)} \cdot \Big( \prod_{\ell=k+1}^{m} {b(Q_{\ell})}^{t_{\ell}} \Big)
\end{equation}
for some constant $c_{3} > 1$.


Let $T$ be a quasisimple group with $T/Z(T)$ not isomorphic to an alternating group of degree larger than ${(\log p)}^3$ and not isomorphic to a simple group of Lie type defined over a field of characteristic $p$. Let $U$ be any finite, faithful $FT$-module over the finite field $F$ of order $p^f$. Put ${|F|}^{s} = |U|$. We claim that 
\begin{equation}
\label{f1}
b(T) = |T/Z(T)| < {(c_{4} \log p)}^{3(s-1)}
\end{equation}
for some universal constant $c_{4}>1$. We use \cite{LS}. A consequence of \cite[(5.3.2), Corollary 5.3.3 and Theorem 5.3.9]{KL} is that if $T/Z(T)$ is a simple group of Lie type in characteristic different from $p$, then $|T/Z(T)| < {(c_{4} \log p)}^{3(s-1)}$ for some constant $c_{4}>1$. By choosing $c_4$ to be at least the maximum of the size of the Monster and the largest value of $r!$ for which $r! \geq r^{r-5}$ where $r$ is a positive integer, our bound on $|T/Z(T)|$ extends to the case when $T/Z(T)$ is a sporadic simple group or $T/Z(T)$ is an alternating group of degree $r$ with $r! \geq r^{r-5}$. If $T/Z(T)$ is an alternating group of degree $r \leq {(\log p)}^{3}$ such that $r! < r^{r-5}$, then $$|T/Z(T)| < r! < r^{r-5} \leq {(\log p)}^{3(r-5)} \leq {(\log p)}^{3(s-1)}$$ where the last inequality follows from \cite[(5.3.2), Corollary 5.3.3 and Proposition 5.3.7]{KL}. This proves our claim. 

For every $\ell$ with $k+1 \leq \ell \leq m$, define $s_{\ell} \geq 2$ by $|U_{\ell,1}| = |F|^{s_{\ell}}$, that is, $s_{\ell} = \dim U_{\ell,1}$. Using (\ref{f1}) and (\ref{eee1}) we find that 
\begin{equation}
\label{f2}
\begin{split}
\prod_{\ell=k+1}^{m} {b(Q_{\ell})}^{t_{\ell}} < \prod_{\ell = k+1}^{m} {(c_{4} \log p)}^{3 (s_{\ell} -1)t_{\ell}} \leq
\prod_{\ell = k+1}^{m} {(c_{4} \log p)}^{3 (s_{\ell}^{t_{\ell}} - 1)} \leq \\ \leq {(c_{4} \log p)}^{3 ((\sum_{\ell = k+1}^{m} s_{\ell}^{t_{\ell}}) - 1)} \leq {(c_{4} \log p)}^{3 ((\prod_{\ell = k+1}^{m} s_{\ell}^{t_{\ell}}) - 1)} \leq {(c_{4} \log p)}^{3(n-1)}. 
\end{split}
\end{equation}

We have $b(H) < {(c_{3} c_{4} \log p)}^{6(n-1)}$ by (\ref{ee4}) and (\ref{f2}). Thus 
$$|H/C| \leq {(c_{2} \log p)}^{n-1} b(H) < {(c_{2}c_{3}c_{4} \log p)}^{7(n-1)}.$$
Finally, set $c_{1} = c_{2} c_{3} c_{4} > 2$. 

This finishes the proof of the lemma in case $V$ is a primitive and irreducible $H$-module. 

Let $H$ be a counterexample to the statement of the lemma with $\dim V$ minimal and with $c_1$ as before. Put $f(p) = {(c_{1} \log p)}^{7}$.

We claim that $V$ must be an irreducible $H$-module. For assume that $V = V_1 \oplus V_2$ where $V_1$ and $V_2$ are non-trivial (completely reducible) $H$-modules. Let $H_1$ be the action of $H$ on $V_1$ and $H_2$ be the action of $H$ on $V_2$. The groups $H_{1}$ and $H_{2}$ are factor groups of $H$ and thus have no non-abelian composition factor which is not a composition factor of $H$. The group $H$ may be viewed as a subgroup of $H_1 \times H_2$. Since $H$ is a counterexample with $\dim V$ minimal, there exist an abelian subgroup $A_1$ in $H_1$ of index at most $f(p)^{m-1}$ and an abelian subgroup $A_2$ in $H_2$ of index at most $f(p)^{n-m-1}$ where $p^{m} = |V_1|$. The group $A = (A_{1} \times A_{2}) \cap H$ is an abelian subgroup of $H$. Moreover, $$|H:A| = |H (A_{1} \times A_{2})|/ |A_{1} \times A_{2}| \leq |H_{1} \times H_{2}|/|A_{1} \times A_{2}| \leq {f(p)}^{n-2} < {f(p)}^{n-1}.$$ This is a contradiction. Thus $V$ is an irreducible $H$-module. 

We claim that $V$ cannot be an imprimitive $H$-module. For let $V = V_{1} + \cdots + V_{t}$ with $t > 1$ be an imprimitivity decomposition for $V$ with each $V_i$ a subspace in $V$ and let $N$ be the normal subgroup of $H$ consisting of all elements leaving every $V_i$ invariant. The group $N$ acts completely reducibly on $V$ and thus also on each $V_i$ by Clifford's theorem. For every $i$ with $1 \leq i \leq t$, let $H_i$ be the action of $N$ on $V_i$. The group $H/N$ may be viewed as a permutation group of degree $t$. In particular, $H$ may be viewed as a subgroup of a full wreath product of the form $W = (H_{1} \times \cdots \times H_{t}):\mathrm{Sym}(t)$. Since $H$ is a counterexample with $\dim V$ minimal, there exists an abelian subgroup $A_i$ in $H_i$, for every $i$ with $1 \leq i \leq t$, such that $|H_{i}:A_{i}| \leq {f(p)}^{(n/t)-1}$. The group $A_{1} \times \cdots \times A_t$ is contained in $W$. Thus $A = (A_{1} \times \cdots \times A_t) \cap N$ is an abelian subgroup in $H$. As before, 
\begin{equation}
\label{ezkell1}
\begin{split}
|N:A| = |N(A_{1} \times \cdots \times A_{t})|/|A_{1} \times \cdots \times A_{t}| \leq |\prod_{i=1}^{t} H_{i}| / |\prod_{i=1}^{t} A_{i}|
\leq \\ \leq \prod_{i=1}^{t} |H_{i}:A_{i}| \leq \prod_{i=1}^{t} {f(p)}^{(n/t)-1} = {f(p)}^{n-t}.
\end{split}
\end{equation}
The permutation group $H/N$ of degree $t$ has no composition factor isomorphic to an alternating group of degree larger than ${(\log p)}^{3}$. It follows that 
\begin{equation}
\label{ezkell2}
|H/N| \leq {(2 \log p)}^{3(t-1)} < f(p)^{t-1}
\end{equation}
by \cite[Corollary 1.5]{attila2002}. We thus have $|H:A| < {f(p)}^{n-t} {f(p)}^{t-1} = {f(p)}^{n-1}$ by (\ref{ezkell1}) and (\ref{ezkell2}). A contradiction.

This finishes the proof of the lemma.  
\end{proof}

Let $X$ be a finite group. Denote the number of orbits of $\mathrm{Aut}(X)$ on $X$ by $k^{*}(X)$. If $X$ acts on a set $Y$, then denote the number of orbits of $X$ on $Y$ by $n(X,Y)$.  

\begin{proposition}
\label{hard}
There exists a universal constant $c_{5} > 0$ such that if $G$ is a finite group having an elementary abelian minimal normal subgroup $V$ of $p$-rank at least $2$ and $|G/V|$ is not divisible by $p^2$, then $k(G) \geq c_{5} p$. 
\end{proposition}  


\begin{proof}

Since $k(G) \geq k(G/V) + n(G,V) - 1$ by Clifford's theorem, it is sufficient to show that $k(G/V) + n(G,V) \geq c_{6} p$ for some universal constant $c_{6} >0$. For this latter claim we may assume that $G/V$ acts faithfully on $V$, that is, $V$ is a faithful and irreducible $H:= G/V$-module. This is because $k(G/V) \geq k(G/C_{G}(V))$ and $n(G,V) = n(G/C_{G}(V),V)$. 

We may assume that $p$ is sufficiently large. 

Every non-abelian (simple) composition factor of $H$ (provided that it exists) has order coprime to $p$, except possibly one which has order divisible by $p$ (but not by $p^2$). There are the following possibilities for a non-abelian composition factor $S$ of $H$: (i) $S$ is an alternating group; (ii) $S$ is a simple group of Lie type in characteristic different from $p$; (iii) $S \cong \mathrm{PSL}(2,p)$; (iv) $S$ is a sporadic simple group. 

Suppose that such a composition factor $S$ exists. We have $k(H) \geq k^{*}(S)$ by \cite[Lemma 2.5]{Pyber}. Since $k^{*}(\mathrm{PSL}(2,p)) \geq (p-1)/4$, by considering diagonal matrices in $\mathrm{SL}(2,p)$, we may exclude case (iii) by choosing $c_5< \frac{1}{5}$ (since we are assuming that $p$ is sufficiently large). Let $S$ be an alternating group of degree $r \geq 5$. Since $|\mathrm{Out}(S)| \leq 4$, we have $k^{*}(S) \geq k(S)/4$. Since $S$ is a normal subgroup of index $2$ in the symmetric group of degree $r$, we have $k(S) \geq \pi(r)/2$ where $\pi(r)$ denotes the number of partitions of $r$. We thus find that $k^{*}(S) \geq c_{7}^{\sqrt{r}}$ for some constant $c_{7} > 1$. If $r > {(\log p)}^{3}$, then $k^{*}(S) > p$ for sufficiently large $p$. Thus we assume that every alternating composition factor of $H$ has degree at most ${(\log p)}^{3}$.  

The group $H$ contains an abelian subgroup $A$ with $|H:A| < |V|^{o(1)}$ as $p \to \infty$, by Lemma \ref{o(1)}. Furthermore, $k(H) \geq k(A)/|H:A| = |A|/|H:A|$ by \cite[p. 502]{Ernest} and $n(G,V) \geq |V|/|H|$. These give 
$$k(H) + n(G,V) \geq \frac{|A|}{|H:A|} + \frac{|V|}{|H|} = \frac{|A|}{|H:A|} + \frac{|V|/|A|}{|H:A|} > \frac{|A| + (|V|/|A|)}{{|V|}^{o(1)}},$$
as $p \to \infty$. Since the real function $g(x) = x + (|V|/x)$ takes its minimum in the interval $[1,|V|]$ when $x = \sqrt{|V|}$, we find that 
$k(H) + n(G,V) > 2 \cdot {|V|}^{(1/2) - o(1)} > p$ for sufficiently large $p$, unless $|V| = p^{2}$. 

Let $|V| = p^{2}$. Note that in \cite[p. 661 and 662]{HK2} it is shown that if $G$ is solvable, we have
$$
k(G/V )+ n(G, V ) -1\geq \frac{49p+1}{60}.
$$
Thus we may assume that $G$ is non-solvable. In this case $H/Z(H)$ is either $\mathrm{Alt}(5)$ or $\mathrm{Sym}(5)$ (given that case (iii) above cannot occur) by \cite[Section XII.260]{Dickson} or \cite[Hauptsatz II.8.27]{Huppert}. Also $|Z(H)| < p$ since $H$ is non-solvable by assumption. Thus there exists a constant $c_{8} > 0$ such that $k(G) \geq n(G,V) \geq |V|/|H| > c_{8} p$.


This finishes the proof of the proposition.
\end{proof}

\section{Finite simple groups}
\label{finite_simple_groups}

In this section we prove Propositions \ref{main_lemma_as} and \ref{lemma_malle}. We first prove a few preliminary lemmas. 

\begin{lemma}
  \label{lemma_factors}
Let $p,q\in\N^{+}\setminus\{1\}$ such that $p\mid q^i+(-1)^a$ and $p\mid q^j+(-1)^b$ for some $i,j\in\N^{+}$ and some $a,b\in\{0,1\}$. If $(i,a)\neq (j,b)$ then $p\leq q^{\min\{i,j,|i-j|\}}+1$.
\end{lemma}

\begin{proof}
  Without loss of generality we may assume that $j\leq i$. By our assumptions, since $p\leq q^j+1$, it is sufficient to show that $p\leq q^{i-j}+1$. We have
$$
p\mid\left(q^i+(-1)^a\right)-\left(q^j+(-1)^b\right)=q^{i}-q^{j}+(-1)^a-(-1)^b.
$$
Assume first that $i\neq j$ and $a=b$. Then $p\mid q^{i}-q^{j}=q^{j}(q^{i-j}-1)$ and so $p\leq q^{i-j}-1$.
If $a\neq b$, then $p\mid \left(q^i+(-1)^a\right)+\left(q^j+(-1)^b\right)=q^{i}+q^{j}=q^{j}(q^{i-j}+1)$
and so $p\leq q^{i-j}+1$. 
\end{proof}

\begin{lemma}
  \label{division_proposition}
  Let $P(x)$ be a polynomial admitting a factorisation of the form $P^{+}(x)P^{-}(x)$ where
  $$
  P^{+}(x)=\prod_{i\in S^{+}}(x^i+1)^{k_i^{+}}
  \quad\text{and}\quad
  P^{-}(x)=\prod_{i\in S^{-}}(x^i-1)^{k_i^{-}}
  $$
  for some sets of integers $S^{+}$ and $S^{-}$ and positive integers $k_i^{+}$ and $k_i^{-}$. Set $m=\max\{S^{+}\cup S^{-}\}$. Assume that $k_i^{+}=k_i^{-}=1$ for every index $i$ strictly larger than $m/2$. If $p$ is an odd prime such that $p^2|P(q)$ for some integer $q\geq 2$, then $p\leq q^{m/2}+1$.
\end{lemma}
\begin{proof}
Let $p$ be an odd prime such that $p^2\mid P(q)$ for some positive integer $q\geq 2$. Let $i\in S^{+}\cup S^{-}$ be such that $p\mid q^i+1$ or $p\mid q^i-1$. If $i\leq m/2$, the assertion is clear, so assume that $i>m/2$. If $p^2\mid q^i+1$ or $p^2\mid q^i-1$, then
$$
p\leq\sqrt{q^i\pm 1}<q^{i/2}+1\leq q^{m/2}+1.
$$
Otherwise there exists $j\in S^{+}\cup S^{-}$ distinct from $i$ such that $p\mid q^j+1$ or $p\mid q^j-1$. By Lemma \ref{lemma_factors}, $p\leq q^{\min\{i,j,|i-j|\}}+1$. Observe that $\min\{i,j,|i-j|\}\leq m/2$. For a proof of this observation we may assume that $i\leq j\leq m$ and so $i$ or $j-i$ is at most $m/2$. 
\end{proof}




\FloatBarrier
\renewcommand{\arraystretch}{1.7}
\begin{table}[ht]
  \begin{tabular}{|c | l | l |}
    \hline
  $S,r$ & $P^{-}(q)$ & $P^{+}(q)$ \\
    \hline
    $A_{n}(q),n$ & $\prod_{i=2}^{n+1}\left(q^i-1\right)$   & $1$ \\
    \hline
      $B_n(q),C_n(q),n$ & $\prod_{i=2}^{n}\left(q^{i}-1\right)$  & $\prod_{i=2}^{n}\left(q^{i}+1\right)$ \\
    \hline
    $D_n(q),n$ & $\prod_{i=1}^{n}\left(q^{i}-1\right)$  & $\prod_{i=1}^{n-1}\left(q^{i}+1\right)$ \\
    \hline
 ${}^2D_n(q),n-1$ & $\prod_{i=1}^{n-1}\left(q^{i}-1\right)$  & $\prod_{i=1}^{n}\left(q^{i}+1\right)$ \\
 \hline
 $G_2(q),2$& $(q-1)(q^3-1)$ & $(q+1)(q^3+1)$ \\
 \hline
 $F_4(q),4$& $(q-1)(q^3-1)^2(q^4-1)$ & \makecell[cl]{$(q+1)(q^3+1)^2(q^4+1)$\\$\cdot(q^6+1)$}\\
 \hline
 $E_6(q),6$&$(q-1)^2(q^3-1)^3(q^5-1)$ & \makecell[cl]{$(q+1)^2(q^2+1)(q^3+1)^3$\\$\cdot(q^4+1)(q^{6}+1)$} \\
 \hline
 $E_7(q),7$& \makecell[cl]{$(q-1)^2(q^3-1)^2(q^5-1)$\\$(q^{7}-1)(q^{9}-1)$} & \makecell[cl]{$(q+1)^2(q^2+1)(q^3+1)^2$\\$\cdot(q^4+1)(q^{5}+1)(q^{6}+1)$\\$\cdot (q^{7}+1)(q^{9}+1)$} \\
 \hline
 $E_8(q),8$ & \makecell[cl]{$(q-1)^2(q^3-1)^2(q^{5}-1)$ \\$\cdot(q^{7}-1)(q^{9}-1)(q^{15}-1)$}& \makecell[cl]{$(q+1)^2(q^2+1)(q^3+1)^2$ \\$\cdot(q^{4}+1)(q^{5}+1)(q^{6}+1)^2$\\$\cdot (q^{7}+1)(q^{9}+1)(q^{10}+1)$\\$\cdot(q^{12}+1)(q^{15}+1)$} \\
 \hline
    \end{tabular}
\caption{}
  \label{table1}
\end{table}
\FloatBarrier

\begin{lemma}
  \label{bound_on_p}
  

Let $S$ be a finite simple group of Lie type of Lie rank $r$ defined over a field of size $q$ as in Table \ref{table1}. If $p$ is an odd prime such that $p\nmid q$ and $p^2\mid|S|$, then $p\leq q^{(r+1)/2}+1$ if $r>8$ and  $p\leq q^{r-\frac{1}{2}}+1$ if $r\leq 8$.
  \end{lemma}

\begin{proof}
  Observe that $|S|$ divides $P^{-}(q)P^{+}(q)$ times a suitable power of $q$, so $p^2\mid P^{-}(q)P^{+}(q)$. If $m$ is as in the statement of Lemma \ref{division_proposition}, then $p\leq q^{m/2}+1$. According to Table \ref{table1}, $m\leq r+1$ if $r> 8$ and $m\leq 2r-1$ if $r\leq 8$. 
  The result follows.
  \end{proof}

\FloatBarrier
    \begin{table}[ht]
      
      \begin{tabular}{|c | l |}
            \hline
  $S,r$ & $P(q)$ \\
    \hline
 ${}^2B_2(q),1$&$(q-1)(q-\sqrt{2q}+1)(q+\sqrt{2q}+1)$\\
  ${}^2G_2(q),1$&$(q-1)(q+1)(q-\sqrt{3q}+1)(q+\sqrt{3q}+1)$\\
    ${}^2F_4(q),2$& \makecell[cl]{$(q-1)^2(q+1)(q-\sqrt{2q}+1)(q+\sqrt{2q}+1)$\\$\cdot (q^{3}+1)(q^3-\sqrt{2q^3}+1)(q^3+\sqrt{2q^3}+1)$}\\
    \hline
\end{tabular}
\caption{}
\label{order_polynomial}
  \label{table2}
\end{table}
    \FloatBarrier

    \begin{lemma}
      \label{lemma_table2}
      Let $S$ be a finite simple group of Lie type of Lie rank $r$ defined over a field of size $q$ as in Table \ref{table2}. There exists a constant $c_{10}$ such that if $p$ is an odd prime with $p\nmid q$ and $p^2\mid|S|$, then $p\leq c_{10}\cdot q^{r-\frac{1}{2}}$.
    \end{lemma}
    \begin{proof}
      Let $P(q)$ be as in Table \ref{table2}. Notice that $p^2|P(q)$. If $p^2$ divides any of the three, four and eight factors in the factorisations of $P(q)$ in Table \ref{table2}, in the respective three cases, then the statement holds. The statement also holds in case $S={}^2F_4(q)$ when $p^2\mid (q-1)^{2}$. We may assume that there are two distinct factors $P_1(q)$ and $P_2(q)$ in the factorisation of $P(q)$ given in Table \ref{table2} which are divisible by $p$. Hence $p\mid P_1(q)-P_2(q)$.

      Let $S$ be ${}^2B_2(q)$ where $q=2^{2t+1}$. In this case $|P_1(q)-P_2(q)|\leq 2\sqrt{2q}$. Similarly, if $S$ is ${}^2G_2(q)$, then $p\leq 2\sqrt{3q}$.

      Let $S$ be ${}^2F_4(q)$ where $q=2^{2t+1}$. Assume first that $P_1(q)$ and $P_2(q)$ have the same degree. In this case $|P_1(q)-P_2(q)|\leq 2\sqrt{2q^3}$. Otherwise $p$ divides a factor of degree $1$, and so $p\leq q+\sqrt{2q}+1$. In any case the result follows.
      \end{proof}

    \FloatBarrier
    \begin{table}[ht]
      
      \begin{tabular}{|c | l |}
        \hline
        $S,r$ & $P(q)$ \\
                            \hline
        ${}^3D_4(q),2$&$(q-1)^2(q+1)(q^2+q+1)(q^3+1)(q^{8}+q^4+1)$\\
    ${}^2E_6(q),4$&$(q^2-1)(q^3-1)^2(q^3+1)^2(q^4-1)(q^4+1)(q^5+1)(q^6+1)(q^9+1)$\\
    ${}^2A_{n-1}(q),[n/2]$&  $\overset{n}{\underset{i \text{ even}}{\underset{i=2}{\prod}}}\left(q^i-1\right)\overset{n}{\underset{i \text{ odd}}{\underset{i=3}{\prod}}}\left(q^i+1\right)$ \\
    \hline
\end{tabular}
\caption{}
\label{order_polynomial2}
  \label{table3}
\end{table}
    \FloatBarrier

    \begin{lemma}
      \label{lemma_nr_conj_exceptions}
      Let $S$ and $r$ be as in Table \ref{table3}. Then both $k({}^3D_4(q))$ and $k({}^2E_6(q))$ are at least $c_{11}\cdot q^{r+2}$ for some constant $c_{11}>0$. We also have $$k({}^2A_{n-1}(q))\geq q^{2r-1}/\min\{2r+1,q+1\}.$$
    \end{lemma}
    \begin{proof}
      See \cite{Luebeck} and \cite[Corollary 3.11]{FG}.
      \end{proof}
    
        \begin{lemma}
      \label{lemma_table3}
      Let $S$ be a finite simple group of Lie type of Lie rank $r$ defined over a field of size $q$ as in Table \ref{table3}. There exists a constant $c_{12}$ such that if $p$ is an odd prime with $p\nmid q$ and $p^2\mid|S|$, then $p\leq c_{12}\cdot q^{r+1}$.
        \end{lemma}
        \begin{proof}
          First let $S$ be ${}^3D_4(q)$. If $p$ divides a factor of $P(q)$ as in Table \ref{table3} of degree at most $3$, then the claim is clear. Otherwise,
          $$
          p^2\mid q^{8}+q^{4}+1=\frac{q^{12}-1}{q^4-1}=\frac{(q^{6}-1)(q^{6}+1)}{q^4-1}.
          $$
          Since $p$ is an odd prime and $p^{2} \mid (q^{6}-1)(q^{6}+1)$, we have $p^2\leq q^6+1$.
          Next, let $S$ be ${}^2E_{6}(q)$. If $p$ divides a factor of $P(q)$ in Table \ref{table3} of degree at most $5$, then the claim follows. Thus we may assume that $p^2|(q^{6}+1)(q^{9}+1)$. By Lemma \ref{division_proposition}, $p\leq q^{4.5}+1$.

          Finally, let $S$ be ${}^2A_{n-1}(q)$. In this case $p\leq q^{r+\frac{1}{2}}+1$, by Lemma \ref{division_proposition}.
        \end{proof}

Let $\mathrm{M}(S)$ denote the Schur multiplier of a non-abelian finite simple group $S$. 

\begin{proposition}
\label{main_lemma_as}
There exists a constant $c_{9} >0$ such that $k^{\ast}(S)\geq c_{9} p$ for any non-abelian finite simple group $S$ and any prime $p$ such that $p^2\mid |S|$ or $p\mid |\operatorname{Out}(S)|$ or $p\mid |\mathrm{M}(S)|$.
\end{proposition}

\begin{proof}
We may assume that $S$ and $p$ are sufficiently large. In particular, we may ignore sporadic simple groups and small alternating groups, and we may assume that $p$ is odd. 

Let $S$ be an alternating group $\mathrm{Alt}(r)$. Since $p$ is odd, $p^2$ must divide $|S|$, and so $p\leq r$. Since there are $\left[r/3\right]$ conjugacy classes of elements of order $3$ in $\mathrm{Sym}(r)$, we have $\left[r/3\right]\leq k^{\ast}(\mathrm{Alt}(r))$, for $r\geq 7$. The constant $c_{9}$ can be chosen such that $\left[r/3\right]\geq c_{9}r$.

  Let $S$ be a finite simple group of Lie type of Lie rank $r$ defined over the field of size $q=\ell^{f}$ for some prime $\ell$ and positive integer $f$. We have
  $$
  k^{\ast}(S)\geq\frac{q^r}{|\mathrm{M}(S)|\cdot|\operatorname{Out}(S)|}
  $$ by \cite[p. 657]{Malle}. Since both $|\operatorname{Out}(S)|$ and $|\mathrm{M}(S)|$ are at most $c_{13}\cdot\min\{r,q\}\cdot f$ for some constant $c_{13}$, we find that
  \begin{equation}
    \label{eq1}
  k^{\ast}(S)\geq\frac{q^r}{(c_{13}\cdot\min\{r,q\}\cdot f)^2}.
  \end{equation}

  From this it follows that if $p\mid q$, then $k^{\ast}(S)\geq c_{14} \cdot p$ for some constant $c_{14} > 0$. Thus assume that $p$ does not divide $q$. Notice that $f\leq \log q$.

  Assume first that $p^2$ does not divide $|S|$. Then $p\leq c_{13} \cdot \min\{r,q\}\cdot\log q$. In order to establish the claim in this case it is sufficient to find a constant $c_{15} > 0$ such that $q^r\geq c_{15}\cdot(c_{13}\cdot\min\{r,q\}\cdot\log q)^3$. For any fixed constant $c_{15}$ this is certainly true for sufficiently large $q$ or sufficiently large $r$. Thus we may assume that $p^{2} \mid |S|$.

  Assume first that $r$ is bounded. Let $S$ be as in Tables \ref{table1} or \ref{table2}. By Lemmas \ref{bound_on_p} and \ref{lemma_table2}, $p$ is at most $c_{16} \cdot q^{r-\frac{1}{2}}$ for some constant $c_{16}$. In this case, by \eqref{eq1}, it is sufficient to find a constant $c_{17}>0$ such that
  $$
  q^r\geq c_{17} \cdot(c_{13}\cdot\min\{r,q\}\cdot\log q)^2\cdot c_{16}\cdot q^{r-\frac{1}{2}}.
  $$
  For any fixed $c_{17}$ this inequality holds apart from at most finitely many pairs $(r,q)$. Next let $S$ be one of the first two groups in Table \ref{table3}. In this case
  $$
  k^{\ast}(S)\geq \frac{c_{18}\cdot q^{r+2}}{(c_{13}\cdot\min\{r,q\}\cdot f)^2}
  $$
  for some constant $c_{18}>0$, by Lemma \ref{lemma_nr_conj_exceptions}. Also, $p\leq c_{12}\cdot q^{r+1}$ by Lemma \ref{lemma_table3}. Again, it is sufficient to find a constant $c_{19}>0$ such that
  $$
  c_{18}\cdot q^{r+2}\geq c_{19} \cdot(c_{13}\cdot\min\{r,q\}\cdot\log q)^2\cdot c_{12}\cdot q^{r+1}.
  $$
  But this is possible since $r$ is bounded.

  Finally, assume that $r$ is unbounded. Let $S$ be as in Table \ref{table1}. By Lemma \ref{bound_on_p}, $p$ is at most $q^{(r+1)/2}+1$. Since there exists a constant $c_{20}>0$ such that
  $$
  q^r\geq c_{20}\cdot(c_{13}\cdot\min\{r,q\}\cdot\log q)^2\cdot (q^{(r+1)/2}+1),
  $$
  the lemma follows by \eqref{eq1}. The only remaining case is $S={}^2A_{n-1}(q)$. Here
  $$
  k^{\ast}(S)\geq \frac{q^{2r-1}}{\min\{2r+1,q+1\}\cdot(c_{13}\cdot\min\{r,q\}\cdot f)^2}
  $$
  by Lemma \ref{lemma_nr_conj_exceptions}. Also, $p\leq c_{12}\cdot q^{r+1}$ by Lemma \ref{lemma_table3}. Again, there exists a constant $c_{21}>0$ such that
  $$
  q^{2r-1}\geq c_{21}\cdot\min\{2r+1,q+1\}\cdot(c_{13}\cdot\min\{r,q\}\cdot\log q)^2\cdot c_{12}\cdot q^{r+1}.
  $$
  The proof is complete with $c_9$ the minimum of $c_{14}$, $c_{15}$, $c_{17}$, $c_{19}$, $c_{20}$ and $c_{21}$.
  \end{proof}

\begin{proposition}
\label{lemma_malle}
There exists a universal positive constant $c_{22}$ such that for every non-abelian finite simple group $S$ and every prime $p$ dividing $|S|$ the inequalities $k(S)\geq c_{22} \cdot |\mathrm{Out}(S)|\cdot \sqrt{p}$ and $k^{\ast}(S)\geq k(S)/|\mathrm{Out}(S)|$ hold. 
\end{proposition}

\begin{proof}
The second inequality follows from \cite[Lemma 2.6]{Pyber}. 

Since $c_{22}$ is allowed to be chosen small enough, it may be assumed that $S$ is different from a sporadic group, different from $\mathrm{Alt}(5)$, $\mathrm{Alt}(6)$, and different from $\mathrm{PSL}(2,16)$, $\mathrm{PSL}(2,32)$ and ${}^2B_2(32)$. 

Let $S$ be an alternating group $\mathrm{Alt}(r)$ with $r \geq 7$. We have 
$$k(S) \geq k^{\ast}(\mathrm{Alt}(r)) \geq c_{9}r \geq c_{9}p$$ from the proof of Proposition \ref{main_lemma_as}. The claimed inequality holds if $c_{22}$ is chosen to be at most $c_{9}/2$. 

Let $S$ be a finite simple group of Lie type of Lie rank $r$ defined over a field of size $q$. Malle in \cite[p. 657]{Malle} showed that $k(S)\geq q^{r}/|\mathrm{M}(S)|$ and 
    $$
  \frac{q^{r}}{|\mathrm{M}(S)|}\geq |\mathrm{Out}(S)|\cdot 2\cdot\sqrt{p-1}
  \geq |\mathrm{Out}(S)|\cdot\sqrt{p}
    $$
    for all $S$ except for $\mathrm{PSL}(2,16)$, $\mathrm{PSL}(2,32)$ and ${}^2B_2(32)$.
\end{proof}

\section{Proof of Theorem \ref{main_theorem}}

Let $G$ be a counterexample to Theorem \ref{main_theorem} with $c = \min \{ c_{5}, c_{9}, c_{22}, c_{22}^{2}/2, 1/2 \}$ and $|G|$ minimal. 

\begin{lemma}
  \label{p_div_normal}
Let $N$ be a non-trivial normal subgroup of $G$. Then $p$ divides $|N|$ and $p^{2}$ does not divide $|G/N|$.
\end{lemma}  

\begin{proof}
The number $k(G)$ of complex irreducible characters of $G$ is at least $k(G/N)$, the number of complex irreducible characters of $G$ with $N$ in their kernel. If $|N|$ is not divisible by $p$, then $|G/N|$ is divisible by $p^{2}$, and so $k(G/N) \geq c p$, since $|G/N| < |G|$.  
\end{proof}

Let $M = \mathrm{soc}(G)$ be the socle of $G$ which is defined to be the product of all minimal normal subgroups of $G$. This group $M$ is a direct product of some of the minimal normal subgroups of $G$ by \cite[Theorem 4.3A (ii)]{DM}. By Lemma \ref{p_div_normal}, we may write $M$ in the form $M_{1} \times M_{2}$ where $M_{1}$ is a (possibly trivial) elementary abelian $p$-group and $M_{2}$ is a (possibly trivial) direct product of non-abelian finite simple groups.

\begin{lemma}
  \label{red_abelian}
The group $M_1$ is trivial or is cyclic of order $p$. 
\end{lemma}

\begin{proof}
  Assume that $p^{2}$ divides $|M_{1}|$. By Lemma \ref{p_div_normal} and Proposition \ref{hard}, we may assume that every abelian minimal normal subgroup of $G$ is cyclic of order $p$. Furthermore, by the minimality of $G$, we may assume that $M = M_{1} = C_{p} \times C_{p}$. Indeed, since $M=C_{p} \times\dots\times C_{p}\times M_{2}$, a factor group of $G$ will have order divisible by $p^2$ unless $M=C_{p}\times C_{p}$.

  We claim that $k(G) \geq k(G/M) + n(G,M) -1 \geq p-1 \geq cp$. For this let $C=C_{G}(M)$ and $H=G/C$. Since $H$ acts faithfully on $M$, it is an abelian group of exponent dividing $p-1$. Let $H_1$ be the kernel of the action of $H$ on the first direct factor $C_p$ of $M$. Then, since $H$ is abelian, $k(G/M)\geq k(H)= |H|$, and we get
 $$
 k(G/M) + n(G,M) - 1
 \geq |H| + n(H/H_{1}, C_{p})\cdot n(H_{1},C_{p}) - 1.$$
Observe that $n(H/H_1,C_p) = 1+\frac{p-1}{|H/H_1|}$ and $n(H_1,C_p) = 1+\frac{p-1}{|H_1|}$. Thus 
$$|H| + n(H/H_{1}, C_{p}) \cdot n(H_{1},C_{p}) - 1 > |H| + \frac{{(p-1)}^{2}}{|H|} \geq p-1.$$  
\end{proof}



\begin{lemma}
  \label{red_non_abelian}
  The group $G$ cannot contain a normal subgroup which is a direct product of $t \geq 2$ copies of a non-abelian finite simple group.
\end{lemma}

\begin{proof}
  Let $N$ be a normal subgroup of $G$ which is a direct product of $t \geq 2$ copies of a non-abelian finite simple group $S$. The prime $p$ divides $|N|$ and therefore $|S|$ by Lemma \ref{p_div_normal}. Since $t\geq 2$ we have $p^2\mid |N|$. On the other hand, by $C_G(N)\cap N=1$ and by the minimality of $G$ we may assume that $C_{G}(N) = 1$. We then have $$N \leq G \leq \mathrm{Aut}(S) \wr \mathrm{Sym}(t).$$

  Let $s = k^{*}(S)$. Choose a representative conjugacy class of $S$ from every $\mathrm{Aut}(S)$-orbit on $S$. Let these be $C_{1}, \ldots , C_{s}$. Put $ C = C_{i_1} \cdots C_{i_{t}}$ where for each $j$ between $1$ and $t$ the integer $i_{j}$ is between $1$ and $s$. Note that $C$ is a conjugacy class of $N$ which can be uniquely labelled by a non-negative integer vector $(r_{1}, \ldots , r_{s})$ where $r_{i}$ ($1 \leq i \leq s$) is the number of $j$ such that $i_{j} = i$ and hence it is contained in a unique conjugacy class of $G$. Note that the conjugation action of $\mathrm{Aut}(S) \wr \mathrm{Sym}(t)$ on $N$ can only fuse $N$-classes which carry the same $(r_{1}, \ldots , r_{s})$ label. Hence we have a family of conjugacy classes of $G$ which are uniquely labelled by these vectors. The set of all such vectors is the set of all non-negative integer solutions to the equation $x_{1} + \cdots + x_{s} = t$. Therefore $k(G) \geq \binom{t+s-1}{t} \geq \binom{s+1}{2} = s(s+1)/2$. Since $s \geq c_{22} \cdot \sqrt{p}$ by Proposition \ref{lemma_malle}, we have $k(G) > (c_{22}^{2}/2) \cdot p$. 
\end{proof}

The group $M_{1}$ is $C_p$ or is trivial and $M_{2}$ is trivial or is a direct product of pairwise non-isomorphic non-abelian finite simple groups, by Lemmas \ref{red_abelian} and \ref{red_non_abelian}.

\begin{lemma}
\label{lemma4.4}
If $M_{2} \not= 1$, then $M_{2}$ is simple.
\end{lemma}

\begin{proof}
Assume that the group $M_{2}$ is non-trivial and not simple. The minimality of $G$ and Lemma \ref{p_div_normal} imply that $M_{2} = S \times F$ where $S$ and $F$ are non-isomorphic non-abelian finite simple groups both of order divisible by $p$. There are at least $k^{*}(S) \cdot k^{*}(F) \geq {(c_{22} \cdot \sqrt{p})}^{2} = c_{22}^{2} \cdot p$ conjugacy classes of $G$ contained in $M_{2}$ by Proposition \ref{lemma_malle}. This is a contradiction. 
\end{proof}


\begin{lemma}
  \label{almost_s}
The group $G$ cannot be almost simple.
\end{lemma}

\begin{proof}
  This follows from Proposition \ref{main_lemma_as}.
  \end{proof}

\begin{lemma}
  \label{red_m2_1}
    We must have $M_{2}=1$.
\end{lemma}

\begin{proof}
Assume for a contradiction that $M_{2} \not= 1$. Then $M_{2}$ is a non-abelian finite simple group $S$ by Lemma \ref{lemma4.4}. Consider the normal subgroup $R = C_{G}(S) \times S$ of $G$.  The group $G/R$ can be considered as a subgroup of $\mathrm{Out}(S)$. Since $C_G(S)$ is normal in $G$, it is either trivial or $p$ divides $|C_{G}(S)|$ by Lemma \ref{p_div_normal}. The first possibility cannot occur by Lemma \ref{almost_s}. Thus $p$ must divide $|C_G(S)|$. On the other hand, $p^2$ cannot divide $|C_G(S)|$ by the minimality of $G$. By a result of Brauer \cite{Brauer}, $k(C_{G}(S)) \geq 2 \sqrt{p-1}$. By Proposition \ref{lemma_malle}, it then follows that
$$
k(G)\geq\frac{k(R)}{|\mathrm{Out}(S)|}
=\frac{k(C_G(S))\cdot k(S)}{|\mathrm{Out}(S)|}
\geq 2\sqrt{p-1} \cdot c_{22} \cdot \sqrt{p} > c_{22} \cdot p.
$$
\end{proof}

Observe that $M=M_1=C_p$ by Lemmas \ref{red_abelian} and \ref{red_m2_1}. Put $C = C_{G}(M)$. Then $|G/C|$ divides $p-1$.
Consider a maximal chain of normal subgroups of $G$ from $C$ to $1$ containing $M$. Let $K_{1}$ be the smallest group in this chain with the property that $p^{2}$ divides $|K_{1}|$. Let $K_{2}$ be the next smaller neighbour of $K_1$ in this chain. The group $M$ is contained in the center of $K_2$ but $|K_{2}/M|$ is not divisible by $p$. By the Schur-Zassenhaus theorem, $K_{2} = M \times K$ for a $p'$-subgroup $K$ of $K_{2}$. Since $K$ is characteristic in $K_{2}$ and $K_{2}$ is normal in $G$, the group $K$ is normal in $G$. This occurs only if $K =1$ by Lemma \ref{p_div_normal}. By the maximality of the chain of normal subgroups of $G$, the group $K_1/M$ is a direct product of isomorphic simple groups $T$. Since $p^2$ divides $|K_1|$, the prime $p$ must divide $|T|$. By the minimality of $G$, the factor group $K_1/M$ is isomorphic to $T$.

\begin{lemma}
  \label{not_cp}
The group $T$ cannot be $C_p$. 
\end{lemma}

\begin{proof}
  Assume that $T$ is cyclic of order $p$. Then $|K_1|=p^2$. Since $k(G)\geq k(G/M)$ and $G$ is a minimal counterexample, we see that $|G/M|$ is not divisible by $p^2$ but divisible by $p$. Thus $G/K_1$ is a $p'$-group. By the Schur-Zassenhaus theorem, there is a $p'$-subgroup $H$ of $G$ such that $G=HK_1$. The group $C_H(K_1)$ is centralized by $K_1$ and it is the kernel of the action of $H$ on the normal subgroup $K_1$ of $G$. Thus $C_H(K_1)$ is normalized by $HK_1=G$. Since $H$ is a $p'$-group, $C_H(K_1)$ must be trivial by Lemma \ref{p_div_normal}. We conclude that $H$ may be considered as an automorphism group of $K_1$. Since $|K_1|=p^2$, the group $H$ and so $G$ must be solvable. The claim follows by \cite{HK}.
  \end{proof}

The group $T$ must be a non-abelian simple group by Lemma \ref{not_cp} and the fact that $p$ divides $|T|$, see the paragraph before Lemma \ref{not_cp}. By Lemma \ref{red_m2_1}, $K_1$ is thus perfect and therefore a quasisimple group.

Notice that $k(G)$ is at least $k^{\ast}(K_1)$. We claim that $k^{\ast}(K_1)\geq k^{\ast}(T)$. Let $\mathcal{T}_1$ and $\mathcal{T}_2$ be two distinct $\operatorname{Aut}(T)$-orbits in $T$. Consider $\phi^{-1}(\mathcal{T}_1)$ and $\phi^{-1}(\mathcal{T}_2)$ where $\phi$ is the natural projection from $K_1$ to $T$. Notice that these two sets are disjoint and $\operatorname{Aut}(K_1)$-invariant. This proves the claim.

The following lemma completes the proof of Theorem \ref{main_theorem}. 

\begin{lemma}
  \label{red_Schur}
  Let $T$ be a non-abelian finite simple group. Let $p$ be a prime divisor of $|T|$ such that $p$ divides the size of the Schur multiplier of $T$. Then $k^{\ast}(T)\geq c_{9} p$.
\end{lemma}
\begin{proof}
  This follows from Proposition \ref{main_lemma_as}.
  \end{proof}

\bigskip

\centerline{\bf Acknowledgement}

\medskip

The authors thank Hung Ngoc Nguyen for a remark on Section 4 and the anonymous referee for a careful reading of the draft. 


  


\end{document}